\documentclass[a4paper, 10pt]{article}

\usepackage[T1]{fontenc}
\usepackage[utf8]{inputenc}

\usepackage{amssymb}
\usepackage{amsmath}
\usepackage{amsfonts}
\usepackage{amsthm}
\usepackage{mathtools}
 
\usepackage{comment}
\usepackage{url}
\usepackage{float}
\usepackage[dvipsnames]{xcolor}
\usepackage{enumerate}
\usepackage{tikz}
\usepackage{caption}
\usepackage{cases}

\theoremstyle{plain}%
\newtheorem{theorem}{Theorem}[section]
\newtheorem{lemma}[theorem]{Lemma}

\newtheorem{corollary}[theorem]{Corollary}

\newcounter{stepcounter}
\newtheorem{step}[stepcounter]{Step}

\theoremstyle{definition}
\newtheorem{definition}[theorem]{Definition}

\theoremstyle{remark}
\newtheorem{remark}[theorem]{Remark}

\usepackage{url}
\usepackage{hyperref}%
\hypersetup{
	colorlinks=false,
	linkbordercolor=gray,
	citebordercolor=gray,
	urlbordercolor=gray}%

\title{An analysis of the induced linear operators associated to divide and color models}

\author{ and Jeffrey E. Steif}

\author{Malin Pal\"o Forsstr\"om
\thanks{Chalmers University of Technology and Gothenburg University, Gothenburg, 
Sweden and
KTH Royal Institute of Technology, Stockholm, Sweden.\ \ Email:
        \hbox{malinpf@kth.se}}
\and         Jeffrey E. Steif
\thanks{Chalmers University of Technology and Gothenburg University, Gothenburg, Sweden.\ \ Email:
        \hbox{steif@chalmers.se}}
}

\date{\today}

\begin{document}

\maketitle

\begin{abstract}
{We study the natural linear operators associated to {\it divide and color (DC) models}. The degree of nonuniqueness of the random partition yielding a DC model is directly related to the dimension of the kernel of these linear operators. We determine exactly the dimension of these kernels as well as analyze a permutation-invariant version. We also obtain properties of the solution set for certain parameter values which will be important in (1) showing that large threshold discrete Gaussian free fields are DC models and in (2) analyzing when the Ising model with a positive external field is a DC model, both in future work.
However, even here, we give an application to the Ising model on a triangle.}

 \newline
 MSC 2010 \emph{subject classifications.}
 Primary  60G99
  \medskip\noindent
\end{abstract}

\section{Introduction, some notation and summary of results}\label{s:intro}

There is a very simple mechanism for constructing random variables with a (positive) dependency 
structure, which are called \emph{divide and color models}. These were introduced in its general form in \cite{st2017}, but have already arisen in many different contexts.

\begin{definition}
A \(\{0,1\}\)-valued process \( X \coloneqq (X_i)_{i \in S} \) is a \emph{divide and color (DC) model} if \( X \) can be generated as follows. First choose a random partition of $S$ according to 
some arbitrary distribution $\pi$, and then independently of this and
independently for different partition elements in the random partition,
assign, with probability \( p \), {\it all} the variables in a partition element
the value \( 1 \) and with probability \( 1-p \) assign 
{\it all} the variables the value \( 0 \). This final \(\{0,1\}\)-valued 
process is called the \emph{DC} model associated to $(\pi,p)$.
We also say that $(\pi,p)$ is a \emph{color representation} of $X$.
\end{definition}

As detailed in \cite{st2017}, many processes in probability theory
are DC  models; examples are the Ising model with zero external field,
the fuzzy Potts model with zero external field, the stationary distributions for the voter model
and random walk in random scenery. 

While certainly the distribution of a divide and color model determines $p$, it 
in fact does not determine the distribution of $\pi$. This was seen
for small sets \( S \) in \cite{st2017}, and this lack of uniqueness
will essentially be completely determined in this paper.

Given a set $S$, we let $\mathcal{B}_S$ denote the collection of partitions of  $S$.
We denote \(\{1,2,3, \ldots, n\}\) by $[n]$ and if \(S =[n] \), we write $\mathcal{B}_n$
for $\mathcal{B}_S$. \( |\mathcal{B}_n| \) is called the \(n \)th Bell number. We let \( P_n \) denote the number of integer partitions on \( n \).
Since $S$ will always be finite, we will, without loss of generality, assume it is equal to \( [n]\) for some $n \in \mathbb{N}$.

The law of any random partition of $[n]$ can be identified with a probability vector 
$q=\{q_\sigma\}_{\sigma\in \mathcal{B}_n}\in \mathbb{R}^{\mathcal{B}_n}$. 
Similarly, the law of any random \(\{0,1\}\)-valued vector $(X_1,\ldots,X_n)$ can be identified with  a probability vector 
$\nu=(\nu(\rho))_{\rho\in \{ 0,1 \}^n}\in \mathbb{R}^{\{ 0,1 \}^n}$. The definition of 
a DC model yields immediately, for each $n$ and $p\in [0,1]$, a map 
$\Phi_{n,p}$ from random partitions of $[n]$,
i.e., from probability vectors $q=\{q_\sigma\}_{\sigma\in \mathcal{B}_n}$
to probability vectors $\nu=(\nu_\rho)_{\rho\in \{ 0,1 \}^n}$.
While a triviality, a crucial observation is that $\Phi_{n,p}$ is an affine map of the relevant simplices.
As a result, this map naturally extends to a linear mapping $A_{n,p}$ from $\mathbb{R}^{\mathcal{B}_n}$ to 
$\mathbb{R}^{\{ 0,1 \}^n}$. This will allow us to more easily analyze questions concerning
nonuniqueness of color representations,  since, by placing the problem in a vector space context,  one can consider formal solutions (to be defined below) which one might show afterwards are in fact nonnegative solutions and therefore color representations.

 To describe the linear operator $A_{n,p}$, we identify $\mathcal{B}_n$ with the natural basis for 
$\mathbb{R}^{\mathcal{B}_n}$ in which case $A_{n,p}$ is uniquely determined by giving the image of each
\( \sigma \in \mathcal{B}_n\) which is done as follows. Given \( \sigma \in \mathcal{B}_n\) and a binary string 
\( \rho \in \{ 0,1 \}^n \), we write
\( \sigma \lhd \rho \) if $\rho$ is constant on the partition elements of $\sigma$. 
We then have
\[
A_{n,p} (\sigma)_\rho\coloneqq A_{n,p} (\rho,\sigma)\coloneqq  \begin{cases}
 p^{c(\sigma,\rho)}(1-p)^{\| \sigma\| - c(\sigma,\rho)} &\textnormal{if } \sigma \lhd \rho \cr
0 &\textnormal{ otherwise}
\end{cases}
\]
where \( \| \sigma \| \) is equal to the number of partition elements in the partition  \( \sigma \) and \( c = c(\sigma, \rho)\) is the number of partition elements on  which $\rho$ is 1.
We would not be surprised if this operator has occurred in other contexts but we have not
been able to find it in the literature.

As seen in~\cite{BMMU}~and~\cite{st2017}, the cases $p=1/2$ and $p\neq 1/2$ behave quite differently when it comes to DC models. We will see this difference also below when studying the dimension of the kernels of the corresponding operators. In \cite{st2017}, the below was obtained for some small values of $n$. For \( \rho \in \{ 0,1\}^n \), we write \( - \rho \) to denote the binary string  where the zeros and ones in $\rho$ are switched, i.e.\ \( -\rho = 1 - \rho \).

\begin{theorem}\label{theorem: MatrixRankp12}
	$A_{n,\frac{1}{2}}$ has rank $2^{n-1}$ and hence nullity $|\mathcal{B}_n|-2^{n-1}$. 
	The range of $A_{n,\frac{1}{2}}$ is
\begin{equation}\label{eq: marg.equal1/2}
\big\{  ( \nu(\rho))_{\rho \in \{0,1\}^n}: \forall \rho, \,\,  \nu(\rho)=\nu(-\rho)\big\}.
\end{equation}
\end{theorem}

\begin{remark}
In the proof of Theorem~\ref{theorem: MatrixRankp12}, we will obtain a concrete formula for a formal solution (to be defined below) for any $\nu$ satisfying~\eqref{eq: marg.equal1/2}. If, in addition, $\nu$ is a probability vector with the property that
the probability of being constant is at least $.5$, then this formal solution will be a
nonnegative solution and hence $\nu$ will be a divide and color model.
We mention that there is no such result  when $p\neq 1/2$.
\end{remark}

\begin{theorem}\label{theorem: MatrixRank}
If $p\not \in \{0, 1/2,1\}$, then $A_{n,p}$ has rank $2^n-n$ and hence nullity $|\mathcal{B}_n|-(2^n-n)$. 
The range of $A_{n,p}$ is equal to
\begin{equation}\label{eq: marg.equal}
\big\{  
( \nu(\rho))_{\rho \in \{0,1\}^n}: \forall i,\,\, 
p\sum_{\rho:\rho(i)=0} \nu(\rho)=(1-p)\sum_{\rho:\rho(i)=1} \nu(\rho) \big\}.
\end{equation}
(The vector subspace defined by~\eqref{eq: marg.equal}  is the vector space analogue of the marginal
distributions each being $p\delta_1 +  (1-p)\delta_0$.)
In particular, if $\nu$ is a probability vector on \( \{ 0,1\}^n \), all of whose marginals 
are $p\delta_1 +  (1-p)\delta_0$, then $\nu$ is in the range of $A_{n,p}$. 
(Of course, there might not be a probability vector \( q = (q_\sigma)_{\sigma \in \mathcal{B}_n}\) which
maps to $\nu$; i.e.\ $\nu$ need not be a DC model.)
\end{theorem}

\medskip

We now discuss the relationship between a nontrivial kernel and nonunique color representations. 
Given $n,p$ and $\nu$, a (not necessarily nonnegative) vector 
$q\in\mathbb{R}^{\mathcal{B}_n}$ is called a {\it formal solution} if
\begin{equation}\label{eq: the main equation system}
 A_{n,p}  {q}=\nu
 \end{equation}
while a nonnegative such vector $q$ is called a {\it nonnegative solution}.
It is easy to see, using inclusion-exclusion, that~\eqref{eq: the main equation system} is equivalent to the system
\begin{equation}\label{eq: the main equation system II}
\sum_{\sigma \in \mathcal{B}_n} p^{\| \sigma_S\| } q_\sigma = \nu_{p}(1^S), \qquad S \subseteq [n]
\end{equation}
where \( \sigma_S \) denotes the restriction of a partition \( \sigma \in \mathcal{B}_n \) to a set \( S \subseteq [n] \), and \( 1^S \) is the event that \( \rho \in \{ 0,1 \}^n\) is equal to 1 on \( S \).
If $\nu$ is a probability vector, then the sum of the coordinates of any formal solution will always be one, but it will be a nonnegative solution if and only if it corresponds to a divide and color representation.
Therefore the relationship between nontriviality of the kernel of  \( A_{n,p} \)  and uniqueness of a color representation (i.e., uniqueness of a nonnegative solution) is as follows.
First, of course \( A_{n,p} \) has a nontrivial kernel if and only if for any $\nu$ in the range,
there are an infinite number of formal solutions. Hence if the kernel is trivial, there is always 
at most one divide and color representation for any DC model. 
The converse is not true since an i.i.d.\ process clearly
has at most one color representation even when the kernel is nontrivial. However, as is also explained in \cite{st2017}, if the kernel is nontrivial, $\nu$ is in the range and there exists a nonnegative solution all of whose coordinates are positive, then
one has infinitely many nonnegative solutions since we can add a small constant times an element in the kernel. 
More generally, 
if $\nu$ is in the range and there exists a nonnegative solution $q$ for $\nu$, then there is another
nonnegative solution (and then infinitely many) if and only if there is an element $q'$ in the 
kernel whose negative-valued coordinates are contained in the support of $q$.

It is sometimes natural to consider situations where one has some further invariance property, a special case being full invariance meaning everything considered is invariant under the full symmetric group. In this case, the characterization of the ranges and of the dimensions of the kernel are given in Theorem~\ref{theorem: MatrixRankINV} in Section~\ref{sec: invariant}.

Our next theorem will be used in \cite{PS.threshold} to show that threshold discrete Gaussian free fields are DC models for large threshold. This theorem is included here since it heavily relies on the 
{\it algebraic picture} used in the proofs of the above results. It gives sufficient conditions for a family of 
probability measures \(  \nu_p  \) to be a DC model for small \( p \) in terms of the asymptotic behavior
of certain probabilities as $p\rightarrow 0$. In the below, for a given set $S$ of coordinates,
\( \nu(1^S) \) will denote the $\nu$-probability that we have all 1's on $S$ and
\(\nu_p(1^S 0 ^{S^c})\) will denote the $\nu$-probability that we have all 1's on $S$ and 0's on $S^c$.

\begin{theorem}\label{theorem: solution sim lemma}
Let \( (\nu_p)_{p \in (0,1)} \) be a family of probability measures on \( \{ 0,1\}^n \). 
Assume that \( \nu_p \) has 
marginals   \(  p\delta_1 + (1-p) \delta_0 \) and that for all \( S \subseteq [n] \) with \( |S| \ge 2 \) 
and all \( k \in S \), as \( p \to 0 \), we have that  
\begin{equation}\label{eq: the main assumption}
p  \nu_p(1^{S \backslash \{ k \}}) \ll \nu_p(1^S) \asymp \nu_p(1^S 0^{[n]\backslash S})
\end{equation} 
and
\[
\lim_{p \to 0}  \sum_{S \subseteq [n] \colon |S| \geq 2} \frac{ \nu_p(1^S0^{S^c}) }{p}< 1 .
  \]
Then \( \nu_p \) is a DC model for all sufficiently small \( p > 0 \).
\end{theorem}

Next, since $p=1/2$ plays a special role, it will turn out to be useful to understand the limiting
behavior of the solution set as $p\rightarrow 1/2$.  This will also be needed in understanding
which Ising models are DC models in the presence of an external field; the latter will be studied
in \cite{palo}. The following result captures this limiting behavior.

\begin{theorem}\label{theorem: lim sol I}
Let \( (\nu_p)_{p \in (0,1)} \) be a family of probability measures on \( \{ 0,1\}^n \). 
Assume that \( \nu_p \) has marginals   \(  p\delta_1 + (1-p) \delta_0 \),  and that for each \( S \subseteq [n] \), \( \nu_p(1^S) \) is differentiable in \( p \) at \( p = 1/2 \). Assume further that for any sets \( T \subseteq [n] \) and \( S \subseteq T \), and any \( p \in (0,1) \), we have that
\[
\nu_p (0^S1^{T \backslash S}) = \nu_{1-p} (0^{T \backslash S} 1^S).
\]
Finally, for each \( p \in (0,1) \) let \( (q_\sigma^{(p)})_{\sigma \in \mathcal{B}_n} \) be a formal solution to the equation 
\begin{equation*}
\sum_{\sigma \in \mathcal{B}_n} p^{\| \sigma_S\| } q_\sigma^{(p)} = \nu_{p}(1^S), \qquad S \subseteq [n].
\end{equation*}
Then the set of  subsequential limits \( (q_\sigma)_{\sigma \in \mathcal{B}_n} \) of sequences \( ((q_\sigma^{(p)})_{\sigma \in \mathcal{B}_n})_{p\in (0,1)} \) 
as \( p \to 1/2 \) is exactly the set of solutions to the system of equations
\begin{equation}\label{eq: limiting system}
\begin{cases}
\sum_{\sigma \in \mathcal{B}_n}  2^{-\| \sigma_S\| } q_\sigma  = \nu_{1/2}(1^S), & S \subseteq [n],\, |S| \text{ even}  \cr
\sum_{\sigma \in \mathcal{B}_n} \| \sigma_S \| 2^{-\| \sigma_S\|+1 } q_\sigma = \nu'_{1/2}(1^S), & S \subseteq [n],\, |S| \text{ odd.} 
\end{cases}
\end{equation}

\end{theorem}

\begin{remark}
It follows from the proof of this theorem that the system of linear equations given by the equations in \eqref{eq: limiting system} corresponding to even sets is equivalent to the linear equation system 
in~\eqref{eq: the main equation system II} for \( p = 1/2 \).
\end{remark}

The following application of Theorem~\ref{theorem: lim sol I} will be proved in Section~\ref{s:p12}.
We consider the Ising model on a triangle with parameters $J$ and $h$; this is the probability
measure on ${\{1,-1\}}^{[3]}$ which has relative weights
$$
e^{J(\sum_{x\neq y}\eta(x)\eta(y))+h\sum_{x}\eta(x)}
$$
to the configuration $\eta$. Call this measure $\nu_{J,h}$. 
For any $J\ge 0$ and $h>0$, by Theorem~\ref{theorem: MatrixRank},
there is a unique $q^{J,h}\in\mathbb{R}^{\mathcal{B}_3}$ with $A_{3,p}  q^{J,h}=\nu_{J,h}$
where $p=p(J,h)$ is chosen to be the probability that a single site is positive.
The uniqueness of this solution also follows from Theorem 2.1(C) in \cite{st2017}.
If we now, for fixed $J$, let $h$ tend to zero, then any subsequential limit $q^J$ of $(q^{J,h})$
necessarily satisfies $A_{3,1/2}  q^J=\nu_{J,0}$. One natural random partition which yields
$\nu_{J,0}$ as its color process is the so-called random cluster model or Fortuin-Kastelyn
representation denoted by $q^{\text{RCM}}$. Interestingly it turns out that
$q^{\text{RCM}}$ does not correspond to the small $h$ limit.
This was first observed by the second author and Johan Tykesson with the help of Mathematica.
Here we obtain it as a direct corollary of Theorem~\ref{theorem: lim sol I}.

\begin{corollary}\label{corollary: Ising}
For all $J> 0$, $\lim_{h\to 0}q^{J,h}$ exists and does not equal $q^{\text{RCM}}$.
\end{corollary}

The rest of the paper is organized as follows. The proofs of 
Theorems~\ref{theorem: MatrixRankp12},~\ref{theorem: MatrixRank} 
and~\ref{theorem: MatrixRankINV}  will be given in 
Section~\ref{s: induced}. Then
Theorem~\ref{theorem: solution sim lemma} will be proved in Section~\ref{s:p0} and 
Theorem~\ref{theorem: lim sol I} as well as Corollary~\ref{corollary: Ising}
will be proved in Section~\ref{s:p12}.

 \section{Dimension of the kernels of the induced linear operators} \label{s: induced}

\subsection{Formal solutions for the \(p=1/2 \) case}\label{theorem: p12}

In this subsection, we prove Theorem~\ref{theorem: MatrixRankp12} and demonstate the statement made
in the remark after the statement of this theorem.

\begin{proof}[Proof of Theorem~\ref{theorem: MatrixRankp12}]

Given a \( \{ 0,1\} \)-symmetric probability vector $\nu = (\nu(\rho))$ on $\{ 0,1\}^n$, 
it is easy to verify (and left to the reader) that a formal 
solution (i.e. a solution to $A_{n,\frac{1}{2}}q= \nu$) is given by
\[
q_{\sigma} = \begin{cases} 2\bigl( \nu({0^S1^{S^c}}) + \nu({1^S0^{S^c}})\bigr), 
&\textnormal{if } \sigma = \{S,S^c\},\,\,  S\neq \emptyset,[n] \cr 1 - \sum_{\sigma' \in \mathcal{B}_n \colon | \sigma'| = 2} q_{\sigma'} &\textnormal{if } \sigma = [n] \cr 0 &\textnormal{otherwise.}\end{cases}
\]
In addition, this yields a color representation (i.e., a nonnegative solution 
to $A_{n,\frac{1}{2}} q=\nu$) if and only if
\begin{equation}\label{eq: constant}
\nu({00\ldots 0}) + \nu({11 \ldots 1 }) \geq 0.5
\end{equation}
by observing that
\begin{align*}
q_{[n]} &= 1 - \sum_{\sigma \in \mathcal{B}_n \mathrlap{\colon} \atop | \sigma| = 2} q_\sigma = 1 - \sum_{\{S,S^c\} \colon \atop S \subseteq [n], \, 0<|S|<n}  
2 \bigl( \nu({0^S1^{S^c}}) + \nu({1^S0^{S^c}})\bigr)
\\&= 1 - 2\Bigl(1-\bigl(\nu({00\ldots 0}) + \nu({11 \ldots 1 })\bigr)\Bigr).
\end{align*}

Clearly every element of the range must satisfy the symmetry condition
\eqref{eq: marg.equal1/2} since $p=1/2$ while the first part of the proof shows that any vector
satisfying \eqref{eq: marg.equal1/2} is in the range. This proves the description of the range and from 
this, it follows immediately that the rank is $2^{n-1}$ and hence the nullity
is $|\mathcal{B}_n|-2^{n-1}$. 
\end{proof}

\subsection{Formal solutions for the  $p\neq 1/2$ case}\label{theorem: pnot12}

In this subsection, we prove Theorem~\ref{theorem: MatrixRank}.

\begin{proof}[Proof of Theorem~\ref{theorem: MatrixRank}]

\begin{step}\label{step I}
The rank of $A_{n,p}$ is at least $2^n-n$. 
\end{step}

\begin{proof}[Proof of Step~\ref{step I}] 
Let \( A'_{n,p} \) be the  $2^n\times |\mathcal{B}_n|$ matrix corresponding to the left hand side of~\eqref{eq: the main equation system II}, i.e.\ let 
\begin{equation*}
	A'_{n,p}(S,\sigma) \coloneqq p^{\|\sigma_S\|},\quad S\subseteq n,\, \sigma \in \mathcal{B}_n.
\end{equation*}
 It suffices
to show that the rank of \( A'_{n,p} \) is at least $2^n-n$.

Let $\sigma^\emptyset$ be the partition into singletons and
for each \( T \subseteq [n] \) with \( |T| > 1 \), let \( \sigma^T \in \mathcal{B}_n \) be 
the unique partition with exactly one non-singleton partition element given by 
\( T\). If e.g. \( n = 5 \) we would have that \( \sigma^{\{ 1,2,3 \}} = (123,4,5) \). 
One easily verifies that  \( \| (\sigma^T)_S \| = |S\backslash T| + (1\land |S\cap T|) \) for $T=\emptyset$
or \( |T| > 1 \).

Consider the equation system 
\[
\nu(1^S) = \sum_{T \subseteq [n] \colon |T| \not = 1}  p^{\| (\sigma^T)_S \|} q_{\sigma^T}, \qquad S \subseteq [n] 
\]
and let \( A''=A_{n,p}'' \) be the corresponding $2^n\times (2^n-n)$ matrix. Define \( B = ( B(S,S'))_{S,S' \subseteq [n]} \) by
\[
B(S,S')  \coloneqq (-p)^{|S|-|S'|} I(S' \subseteq S) .
\]
If we order the rows (from top to bottom) and columns (from left to right) of \( B \) such that the 
sizes of the corresponding sets are increasing, then \( B \) is a lower triangular matrix with 
\( B(S,S) = 1 \) for all \( S \subseteq [n] \). In particular, this implies that \( B \) is invertible for 
all \( p \in (0,1) \), and hence  \( A'' \) and \( BA'' \) (also a $2^n\times (2^n-n)$ matrix) have the 
same rank.
Moreover, for any \( S,T \subseteq [n] \) with \(  |T| \not = 1 \)  we get  
\begin{align*}
(BA'')(S,T) &= \sum_{S' \colon S' \subseteq S} (-p)^{|S|-|S'|}A''(S',T) 
\\&= \sum_{S' \colon S' \subseteq S  } (-p)^{|S|-|S'|}p^{|S'\backslash T| + (1\land |S' \cap T|) } 
\\&=(-p)^{|S|} \sum_{S' \colon S' \subseteq S  } (-p)^{-|S'|}p^{|S'\backslash T| } p^{1\land |S' \cap T|} 
\\&=(-p)^{|S|} \sum_{S' \colon S' \subseteq S  } (-p)^{-|S'\cap T|}(-1)^{|S'\backslash T| } p^{1\land |S' \cap T|} 
\\&=
 \begin{cases}
 (-p)^{|S|} \sum_{S' \colon S' \subseteq S  } (-p)^{-|S'|} p^{1\land |S' |}  &\text{if } S \subseteq T \cr
 0 &\text{otherwise.}
\end{cases}
\end{align*}
In the case \( S \subseteq T \), we can simplify further to obtain
\begin{align*}
(BA'')(S,T) &= 
 (-p)^{|S|} \sum_{S' \colon S' \subseteq S } (-p)^{-|S'|} p^{1\land |S' |}   
 \\&
= (-p)^{|S|} \left( (1 - p^{-1})^{|S|} \cdot p + (1-p) \right) 
 \\&
=p  (1-p)^{|S|}+ (-p)^{|S|}(1-p).
\end{align*}
Note that since $p\neq 1/2$, if \( S \subseteq T \), then \( (BA'')(S,T) = 0 \) if and only if \( |S| = 1\).
If we order the rows (from top to bottom) and columns (from left to right) of \( BA'' \) so that the 
corresponding sets are increasing in size, it is obvious that the \((2^n -n) \times (2^n -n) \)
submatrix of \( BA'' \) obtained by removing the rows corresponding to $|S|=1$ has full rank.
This implies that \( BA'' \) has rank at least \( 2^n -n \) which implies the same for
\( A'' \) since $B$ is invertible. Finally, since \( A'' \) is a submatrix of \(A'_{n,p} \), we obtain
the desired lower bound on the rank of the latter.
\end{proof}

\begin{step}\label{step II}
The rank of $A_{n,p}$ is at most $2^n-n$. 
\end{step}

\begin{proof}[Proof of Step~\ref{step II}]
We first claim that if $ \nu = ( \nu({\rho}))_{\rho \in \{ 0,1 \}^n}$ is in the range, then it is in the set
defined in~\eqref{eq: marg.equal}. To see this, let 
$\nu=A_{n,p} \, q$  for some \( q = (q_\sigma)_{\sigma \in \mathcal{B}_n} \) and fix an $i \in [n]$. The expression
in the left hand side of~\eqref{eq: marg.equal} becomes
$$
p\sum_{\rho:\rho(i)=0} \sum_{\sigma \in \mathcal{B}_n}A_{n,p} (\rho,\sigma)q_\sigma=
p\sum_{\sigma \in \mathcal{B}_n}q_\sigma \sum_{\rho:\rho(i)=0} A_{n,p} (\rho,\sigma).
$$
With $\sigma \in \mathcal{B}_n$ fixed, let 
$$
T^i_\sigma: \{\rho:\rho(i)=0\} \mapsto \{ \rho:\rho(i)=1\}
$$
be the bijection which flips $\rho$ on the partition element of $\sigma$ which contains $i$.
It is clear that for all $\rho$ with $\rho(i)=0$, we have
$$
(1-p)A_{n,p} (T^i_\sigma(\rho),\sigma) =p A_{n,p} (\rho,\sigma) 
$$
and hence the previous expression is
$$
(1-p) \sum_{\sigma\in \mathcal{B}_n}q_\sigma \sum_{\rho:\rho(i)=0} A_{n,p} (T^i_\sigma(\rho),\sigma)=
(1-p)\sum_{\sigma\in \mathcal{B}_n}q_\sigma \sum_{\rho:\rho(i)=1} A_{n,p} (\rho,\sigma)=
$$
$$
(1-p)\sum_{\rho:\rho(i)=1} \sum_{\sigma \in \mathcal{B}_n} A_{n,p} (\rho,\sigma) q_\sigma=
(1-p)\sum_{\rho:\rho(i)=1} \nu(\rho).
$$

$A_{n,p}$ is mapping into a $2^n$-dimensional vector space and each of the $n$ equations
in~\eqref{eq: marg.equal} gives one linear constraint. It is easy to see that these
$n$ constraints are linearly independent (for example, one can see this by just looking
at the number of times each of the vectors $0^k1^{n-k}$ appears on the two sides).
It follows that  the rank of $A_{n,p}$ is at most $2^n-n$. 
\end{proof}

With Steps 1 and 2 completed, together with the claim at the start of Step~2,
we conclude that the rank is as claimed and the range is characterized as claimed.
Finally, the claim concerning probability vectors follows immediately.

\end{proof}

\begin{remark}\textcolor{white}{.}
\begin{enumerate}[(i)]
\item The argument for the $p\neq 1/2$ case can equally well be carried out with minor modifications
for the $p=1/2$ case but we preferred the simpler argument which even gives more. 
\item This last proof shows that, when dealing with formal solutions, we only need to use partitions
which have at most one nonsingleton partition element. This is in large contrast to the earlier proof of
the $p=1/2$ case where we only needed to use partitions which have at most two partition elements. 
\item The rank of an operator as a function of its matrix elements
is not continuous but it is easily seen to be lower semicontinuous.
We see this lack of continuity at $p=1/2$ as well as of course at $p=0 $ and \( p = 1 \).
\end{enumerate}
\end{remark}

\subsection{The fully invariant case}\label{sec: invariant}

It is sometimes natural to consider situations where one has some further invariance property.
One  natural case is the following. 
The symmetric group $S_n$ acts naturally on $\mathcal{B}_n$, 
$\{ 0,1\}^n$, $\mathcal{P}(\mathcal{B}_n)$, 
$\mathcal{P}(\{ 0,1\}^n)$, $\mathbb{R}^{\mathcal{B}_n}$ and $\mathbb{R}^{\{ 0,1 \}^n}$
where $\mathcal{P}(X)$ denotes the set of probability measures on $X$.
(Of course $\mathcal{P}(\mathcal{B}_n)\subseteq \mathbb{R}^{\mathcal{B}_n}$ and the action on the former is just
the restriction of the action on the latter; similarly for 
$\mathcal{P}(\{ 0,1\}^n)\subseteq\mathbb{R}^{\{ 0,1 \}^n}$.)
To understand uniqueness of a color representation when we restrict to $S_n$-invariant probability
measures, it is natural to again extend to the vector space setting, which is done as follows.
Let $Q^{\rm{Inv}}_n\coloneqq  \{q\in \mathbb{R}^{\mathcal{B}_n}: g(q) = q \,\,\,\forall g\in S_n\}$
and $V^{\rm{Inv}}_n\coloneqq  \{\nu\in \mathbb{R}^{\{ 0,1 \}^n}: g(\nu)= \nu \,\,\,\forall g\in S_n\}$.
We next let $A^{\rm{Inv}}_{n,p}$ be the restriction of $A_{n,p}$ to  $Q^{\rm{Inv}}_n$.
It is elementary to check that $A^{\rm{Inv}}_{n,p}$ maps into $V^{\rm{Inv}}_n$ and 
furthermore, it is easy to check, by averaging, that 
\begin{equation}\label{eq: InvariantImage}
A^{\rm{Inv}}_{n,p}(Q^{\rm{Inv}}_n)= A_{n,p}(\mathbb{R}^{\mathcal{B}_n})\cap V^{\rm{Inv}}_n.
\end{equation}
Recalling that $P_n$ is the set of partitions of the integer $n$,
we have an obvious mapping from $\mathcal{B}_n$ to $P_n$, denoted by $\sigma\mapsto\pi(\sigma)$, which is constant on \( S_n \) orbits.
$\mathbb{R}^{P_n}$ can then be canonically identified with $Q^{\rm{Inv}}_n$ via  
$(q_{\pi})_{\pi \in P_n}$ is identified with $(q_{\sigma})_{\sigma \in \mathcal{B}_n}$ 
where $q_{\sigma}={q_{\pi(\sigma)}}/{a_{\pi(\sigma)}}$ where $a_{\pi}$
is the number of $\sigma$'s for which $\pi(\sigma)=\pi$. 
In an analogous way, $V^{\rm{Inv}}_n$ can be canonically identified with $\mathbb{R}^{n+1}$;  namely, $(\nu_i)_{0\le i \le n}$ is identified with $(\nu({\rho}))_{\rho \in \{ 0,1\}^n}$ 
where $\nu({\rho})={\nu_{\|\rho\|}}/{\binom{n}{\|\rho\|}}$ and \( \| \rho \| \) is the number of ones in the binary string \( \rho \).

Using this notation, we have the following theorem. Again, in \cite{st2017}, this was done for some small values of $n$.

 \begin{theorem}\label{theorem: MatrixRankINV}
(i). For $p\not \in \{0, {1}/{2},1\}$, 
$A^{\rm{Inv}}_{n,p}$ has rank $n$ and hence nullity $|P_n|-n$. 
The range of $A^{\rm{Inv}}_{n,p}$ (after identifying $V^{\rm{Inv}}_n$ with $\mathbb{R}^{n+1}$) is
\begin{equation}\label{eq: marg.equalINVp}
\big\{  (\nu_0,\ldots,\nu_n): \nu_n=\frac{p}{1-p}\sum_{k=0}^{n-1}   \frac{n-k}{n} \nu_k
-\sum_{k=0}^{n-2}  \frac{k+1}{n}  \nu_{k+1} \big\}.
\end{equation}

(ii) $A^{\rm{Inv}}_{n,\frac{1}{2}}$ has rank $\lfloor n/2 \rfloor +1$ and hence nullity 
$|P_n|-\lfloor n/2 \rfloor -1$. 
The range of $A_{n,\frac{1}{2}}$ (after identifying $V^{\rm{Inv}}_n$ with $\mathbb{R}^{n+1}$) is
\begin{equation}\label{eq: marg.equalINV1/2}
\big\{   (\nu_0,\ldots,\nu_n): \nu_i= \nu_{n-i} \,\,\, \forall i=1,\ldots,n \big\}.
\end{equation}
\end{theorem}

\begin{proof}
(i). Denoting by $U_n$ the subset of $\mathbb{R}^{n+1}$ satisfying~\eqref{eq: marg.equalINVp}, we claim that
(after identifying $V^{\rm{Inv}}_n$ with $\mathbb{R}^{n+1}$)
\begin{equation}\label{eq: Range}
U_n= A^{\rm{Inv}}_{n,p}(Q^{\rm{Inv}}_n).
\end{equation}
Since $U_n$ is clearly an $n$-dimensional subspace of $\mathbb{R}^{n+1}$, 
the proof of (i) will then be done. 
To see this, first take $\nu \in U_n$ and let $\nu^{\rm{Inv}}$ be the corresponding element in
$V^{\rm{Inv}}_n$. We first need to show that \eqref{eq: marg.equal} is satisfied for $\nu^{\rm{Inv}}$. 
Fixing any $i \in [n]$, we have
$$
p\sum_{\rho:\rho(i)=0} \nu^{\rm{Inv}}_{\rho}=p\sum_{k=0}^{n-1} \binom{n-1}{k}\frac{\nu_k}{\binom{n}{k}}=
p\sum_{k=0}^{n-1} \frac{n-k}{n} \nu_k
$$
and
\begin{align*}
&(1-p)\sum_{\rho:\rho(i)=1} \nu^{\rm{Inv}}_{\rho} =
(1-p)\sum_{k=0}^{n-2} \binom{n-1}{k}\frac{\nu_{k+1}}{\binom{n}{k+1}}+(1-p)\nu_n
\\&\qquad =
(1-p)\sum_{k=0}^{n-2} \frac{k+1}{n} \nu_{k+1}+(1-p)\nu_n.
\end{align*}
Hence, since $\nu \in U_n$, \eqref{eq: marg.equal} holds. In view of~\eqref{eq: InvariantImage}, 
this shows $\subseteq$ in~\eqref{eq: Range} holds. 

Now fix $ \nu^{\rm{Inv}}\in A^{\rm{Inv}}_{n,p}(Q^{\rm{Inv}}_n)$. Clearly 
$ \nu^{\rm{Inv}}\in V^{\rm{Inv}}_n$ and by 
Theorem~\ref{theorem: MatrixRank}, \eqref{eq: marg.equal} holds. The above computation shows
that the corresponding $\nu \in \mathbb{R}^{n+1}$ 
satisfies~\eqref{eq: marg.equalINVp} and hence is in $U_n$.
This shows that $\supseteq$ in~\eqref{eq: Range} holds as well.

(ii). Denoting now by $U_n$ the subset of $\mathbb{R}^{n+1}$ 
satisfying~\eqref{eq: marg.equalINV1/2}, we claim that
\begin{equation}\label{eq: Range1/2}
U_n= A^{\rm{Inv}}_{n,\frac{1}{2}}(Q^{\rm{Inv}}_n).
\end{equation}
Since $U_n$ is clearly an $(\lfloor n/2 \rfloor +1)$-dimensional subspace of 
$\mathbb{R}^{n+1}$, the proof 
of (ii) will then be done. However, in view of~\eqref{eq: marg.equal1/2} in 
Theorem~\ref{theorem: MatrixRank} and~\eqref{eq: InvariantImage}, this is immediate.
\end{proof}

\section{Limiting solutions as $p$ approaches $0$} \label{s:p0}

In this section, we provide a proof of Theorem~\ref{theorem: solution sim lemma}.

\begin{proof} [Proof of Theorem~\ref{theorem: solution sim lemma}]
We will show that given the assumptions of the lemma, for \( p>0\) sufficiently small there is a color representation \( (q_\sigma) = (q_\sigma(p)) \) of \( X_p \sim \nu_p \) which is such that \( q_\sigma = 0 \) for all \(  \sigma \in \mathcal{B}_n\) with more than one non-singleton partition element. 
To this end, fix \( p \in (0,1/2) \). We now refer to the proof of Theorem~\ref{theorem: MatrixRank}.
By Step 1 in that proof we have that a color representation \( (q_\sigma(p)) \) with the desired properties exists  if and only if the (unique) solution \( (q_{\sigma^S}(p))_{|S| \not = 1} \) to
\begin{equation}\label{eq: linear equation sysem goal}
\nu_p(1^S) = \sum_{T \subseteq [n] \colon |T|\not = 1}  p^{\| (\sigma^T)_S \|} q_{\sigma^T}(p), \qquad S \subseteq [n] \colon |S| \not = 1
\end{equation}
is non-negative.
As in the proof of  Theorem~\ref{theorem: MatrixRank}, let \( A'' \) be the \( 2^n \times  (2^n-n)\) matrix corresponding to~\eqref{eq: linear equation sysem goal} and define \( B = ( B(S,S'))_{S,S' \subseteq [n]} \) by
\[
B(S,S')  \coloneqq (-p)^{|S|-|S'|} I(S' \subseteq S) .
\]
In the proof of Step 1 of Theorem~\ref{theorem: MatrixRank}, we saw that for \( S,T \subseteq [n] \) with \( |T| \not = 1 \),
\begin{align*}
(BA'')(S,T) &= 
 \begin{cases}
p(1-p)^{|S|} + (-p)^{|S|} (1-p) &\text{if } S \subseteq T \cr
 0 &\text{otherwise.}
\end{cases}
\end{align*}
Let \( D = (D(S,S'))_{S,S' \subseteq [n]}\) be the diagonal matrix with 
\[ D(S,S) \coloneqq (p(1-p)^{|S|} + (-p)^{|S|} (1-p))^{-1} I(|S| \not = 1) ,\quad S \subseteq [n].
\]
 Then for \( S,T \subseteq [n] \) with \( |T| \not = 1 \),
\begin{align*}
(DBA'')(S,T) = I(S \subseteq T)  \cdot I(|S| \not = 1).
\end{align*}
Furthermore, one can verify that if we define the matrix \( C = (C(S,S'))_{S,S' \subseteq [n]} \)  by 
\[
C(S,S') \coloneqq 
\begin{cases}
 (-1)^{|S'|-|S|} I(S \subseteq S')  &\text{if } |S| \geq 2 \text{ or } S=S'=\emptyset \cr
  (-1)^{|S'|-|S|} I(S \subseteq S') \cdot(1-|S'|) &\text{if }S' \not = S = \emptyset \cr
  0 &\text{otherwise}
\end{cases}
\]
then (since \( p \not = \{ 0,1/2,1 \} \))
\begin{equation}\label{eq: ABCD eq}
(CDBA'')(S,T) = I(S=T) \cdot I(|S| \not = 1).
\end{equation}
Since, by  Step 1 in the proof of Theorem~\ref{theorem: MatrixRank}, the rank of \(A'' \) is exactly \( 2^n-n \), it follows that if we think of \( \nu_p \) as a column vector, then~\eqref{eq: linear equation sysem goal} is equivalent to
\begin{equation}\label{eq: matrix multiplication thing}
q_{\sigma} (p)=   \begin{cases} e_{S}^t CDB \nu_p &\text{if } \sigma = \sigma^S,\, |S| \not = 1 \cr 0 &\text{otherwise}\end{cases}
\end{equation}
(with \( t \) here meaning transpose and \( e_S \) denoting the vector \( (I(S'=S))_{{S'} \subseteq [n]} \)).
Now note that   \( DB\nu_p(1^\emptyset) =  \nu_p(1^\emptyset) \) and that if \( S \subseteq [n] \) has size \( |S| \geq 2 \), we have that
\[
DB\nu_p(1^S) = e_{S}^T DB \nu_p  = \frac{\sum_{S' \colon S' \subseteq S} (-p)^{|S|-|S'|} \nu_p (1^{S'})}{p(1-p)^{|S|}+(-p)^{|S|}(1-p)}.
\]
Since \( |S|\geq 2 \), the denominator is  \(p (1 + O(p))\), and 
by the left hand side of~\eqref{eq: the main assumption} the numerator is given by \(  \nu_p(1^S) + o(\nu_p(1^S))  \).
It follows that 
\begin{align*}
&DB\nu_p(1^S)  = p^{-I(|S|\geq 2)} (\nu_p(1^S)+o(\nu_p(1^S))) (1 + O(p))
 \\&\qquad = p^{-I(|S|\geq 2)} \nu_p(1^S) 
 +   o(p^{-I(|S|\geq 2)} \nu_p(1^S) )  
\end{align*}
 for any \( S \subseteq [n] \) with \( |S| \not = 1 \). 
If we apply \( C \) to the vector \(  (DB\nu_p(1^S))_{C \subseteq [n]}\), a computation  shows that we get  
\[ e_s^t CDB \nu_p =   p^{-1}  \nu_p(1^S0^{[n] \backslash S})) + o(p^{-1} \nu_p(1^S) )  , \quad S \subseteq [n], \, |S| \geq 2.
\]
 By~\eqref{eq: ABCD eq} and the assumption that \( \nu_p(1^S) \asymp \nu_p(1^S 0^{S^c})\), it follows that \( q_{\sigma^S} \sim p^{-1}  \nu_p(1^S0^{[n] \backslash S})\) for any \( S \subseteq [n] \) with \( |S| \geq 2 \).
Since \( q_{\sigma^\emptyset} = 1 - \sum_{S \subseteq [n] \colon |S| \geq 2} q_{\sigma^S}\), again using the assumptions, this  concludes the proof. 

\end{proof}

\section{Limiting solutions as $p$ approaches $1/2$} \label{s:p12}

Before we proceed to the proof of Theorem~\ref{theorem: lim sol I}, we state and prove a 
few lemmas that will be useful in this proof.

\begin{lemma}\label{lemma: invertible operation}
Let \( f \colon 2^{[n]} \to \mathbb{R} \). Define \( \varphi f \colon 2^{[n]} \to \mathbb{R} \) by
\[
\varphi f(S) \coloneqq \sum_{S' \colon S' \subseteq S} (-2)^{|S'|-|S|} f(S'), \quad S \subseteq [n]
\]
and \( \varphi^{-1} f \colon 2^{[n]} \to \mathbb{R} \)  by
\[
\varphi^{-1} f(S) \coloneqq \sum_{S' \colon S' \subseteq S} 2^{|S'|-|S|} f(S'), \quad S \subseteq [n].
\]
Then
\[
\varphi^{-1}\varphi f(S) = f(S), \quad S \subseteq [n].
\]
\end{lemma}

This lemma is a type of M\"obius inversion formula. For completeness, we present a short proof.

\begin{proof} 
Let \( T \subseteq [n] \). Then we have that
\begin{align*}
&\varphi^{-1} \varphi f(T) =  \sum_{S \colon S \subseteq T}  2^{|S|-|T|}  \varphi f(S)
=
 \sum_{S \colon S \subseteq T}  2^{|S|-|T|}  \sum_{S' \colon S' \subseteq S} (-2)^{|S'|-|S|}f(S')
\\&\qquad =
\sum_{S' \colon S' \subseteq T}  \sum_{S \colon S' \subseteq S \subseteq T}  2^{|S|-|T|}   (-2)^{|S'|-|S|}f(S')
\\&\qquad =
2^{-|T|}\sum_{S' \colon S' \subseteq T} 2^{|S'|} f(S') \sum_{S \colon S' \subseteq S \subseteq T}    (-1)^{|S'|-|S|} 
\\&\qquad =
2^{-|T|}\sum_{S' \colon S' \subseteq T} 2^{|S'|} f(S') \sum_{S'' \colon S'' \subseteq  T\backslash S'}    (-1)^{|S''|} 
\\&\qquad =
2^{-|T|}\sum_{S' \colon S' \subseteq T} 2^{|S'|} f(S') I(S' = T)
= f(T).
\end{align*} 
\end{proof}

\begin{lemma}\label{lemma: full rank matrix}
Define \( A \colon  \mathcal{B}_n \to \mathbb{R}^{2^{[n]}}   \) by
\[
A(S,\sigma) \coloneqq I(\sigma_S \text{ has most one odd sized partition element}).
\]
Then \( A \) has rank \( 2^n - n \).
\end{lemma}

\begin{proof}[Proof of Lemma~\ref{lemma: full rank matrix}]
Recall the definition of \( \sigma^T \) from the proof of Theorem~\ref{theorem: MatrixRank}. One can check that for any \( S \subseteq [n] \),
\[
A(S,\sigma^T) = 
\begin{cases}
1 &\text{if } S \subseteq T \cr
1 &\text{if } |S| \text{ is odd and } |S \backslash T| = 1 \cr
0 &\text{else}.
\end{cases}
\]
This implies in particular that
\[
A(S,\sigma^T) - \sum_{i \in S} A(S\backslash \{ i \}, \sigma^T) \, I(|S| \text{ is odd}) = \bigl(1 - |S| \cdot I(|S| \text{ is odd}) \bigr) \cdot I(S \subseteq T).
\]
Since \( (I(S \subseteq T))_{S,T \subseteq [n]} \) has full rank, it follows that \( A \), when restricted to sets \( S \subseteq [n] \) with \( |S| \not = 1 \), has full rank, i.e.\ rank \(2^n-n \). Since \( A(\{ i \},\sigma^T) = A(\emptyset, \sigma^T) = 1 \) for all \( T \subseteq [n] \) with \( |T| \not = 1 \) and all \( i \in [n] \), \( A \) can have rank at most \( 2^n-n \), hence the desired conclusion follows.
\end{proof}

\begin{lemma} \label{eq: odd rows are zero}
If \( S \subseteq [n] \),  \( |S| \) is odd and   \( \nu \colon \{ 0,1\}^n \to \mathbb{R} \) is \( \{ 0,1\} \)-symmetric, 
\[
\sum_{T \colon T\subseteq S} (-2)^{|T|} \nu(1^{T}) = 0.
\]
\end{lemma}

\begin{proof}[Proof of Lemma~\ref{eq: odd rows are zero}]
Fix a set \( S \subseteq [n] \) with \( |S| \) odd.
Since \( |S|\) is odd and \( \nu \) is symmetric, 
\begin{align*}
&0 = \sum_{\rho \in \{ 0 ,1 \}^n}(-1)^{\sum_{i \in S} \rho(i)} \nu(\rho) = \sum_{ T \colon T \subseteq S} (-1)^{|S \backslash T|} \nu(0^T 1^{S \backslash T})
\\&\qquad
= \sum_{T \colon T \subseteq S} (-1)^{|S  | - |T|} \nu(0^T 1^{S \backslash T}).
\end{align*}
Next, by inclusion exclusion, for any set \( T \subseteq S \),  
\[
\nu(0^T 1^{S \backslash T}) = \sum_{T' \colon T' \subseteq T} (-1)^{|T'|} \nu(1^{T' \cup (S \backslash T)}).
\]
Combining the two earlier equations and then changing the order of summation, we obtain
\begin{align*}
0  
&= \sum_{ T \colon T \subseteq S} (-1)^{|S  | - |T|} \sum_{T' \colon T' \subseteq T} (-1)^{|T'|} \nu(1^{T' \cup (S \backslash T)})
\\&=
\sum_{T',T \colon  T' \subseteq T \subseteq S} (-1)^{  |S | - (|T| - |T'|) } \nu(1^{S \backslash ( T \backslash T')})
\\&=
\sum_{ S' \colon S' \subseteq S }  (-1)^{|S'|} \nu(1^{S'}) \cdot 2^{|S'|}
=
\sum_{S' \colon  S' \subseteq S}  (-2)^{|S'|} \nu(1^{S'})
\end{align*}
which is the desired conclusion.
\end{proof}

\begin{lemma} \label{eq: even rows are zero}
Suppose that \( S \subseteq [n] \),  \( |S| \) is even and that \( \nu_p \colon \{ 0,1\}^n \to \mathbb{R} \) is differentiable in \( p \) at \( p =1/2 \). Suppose further that for all \( T \subseteq S  \) and all \( p \in   (1/2,1) \), \( \nu_p \) satisfies
\[
\nu_p(1^S0^{T\backslash S}) = \nu_{1-p}  (1^{T \backslash S} 0^S ).
\]
Then
\[
\sum_{ T \colon T\subseteq S} (-2)^{|T|} \nu_{1/2}'(1^{T}) = 0.
\]
\end{lemma}

\begin{proof}[Proof of Lemma~\ref{eq: even rows are zero}]
Fix a set  \( S \subseteq [n] \) with \( |S| \) even.
Note  that, using  the assumption on   \( (\nu_p) \),  for any \( T \subseteq S \), we have that
\begin{equation}\label{eq: derivative equality}
\begin{split}	
&\nu_{1/2}'(0^T 1^{S \backslash T}) = \lim_{p \to 1/2} \frac{\nu_p(0^T 1^{S \backslash T})-\nu_{1-p}(0^T 1^{S \backslash T})}{p - (1-p)}
\\&\qquad  = \lim_{p \to 1/2} \frac{\nu_{1-p}(1^T 0^{S \backslash T})-\nu_{p}(1^T 0^{S \backslash T})}{p - (1-p)}
= -\nu_{1/2}'(1^T 0^{S \backslash T}).
\end{split}
\end{equation}
Next, by the proof of the Lemma~\ref{eq: odd rows are zero}, we have that
\[
2\sum_{T \colon T \subseteq S}  (-2)^{|T|} \nu'_{1/2}(1^{T})
=
 2\sum_{T \colon  T \subseteq S} (-1)^{|S  | - |T|} \nu'_{1/2}(0^T 1^{S \backslash T}).
\]
By~\eqref{eq: derivative equality}, this equals
\begin{align*}
& \sum_{T \colon  T \subseteq S} (-1)^{|S  | - |T|} \Bigl[ \nu'_{1/2}(0^T 1^{S \backslash T}) - \nu'_{1/2}( 0^{S \backslash T} 1^T ) \Bigr]
\\&\qquad  =
  \sum_{T \colon  T \subseteq S} \nu'_{1/2}(0^T 1^{S \backslash T})  
 \Bigl[ (-1)^{|S  | - |T|} - (-1)^{|S  | - |S \backslash T|}  \Bigr].
\end{align*}
Since \( |S| \) is even,   \( |T| \) and \( |S\backslash T| \) have the same parity, and hence the desired conclusion follows.

\end{proof}

We now proceed to the proof of Theorem~\ref{theorem: lim sol I}.
\begin{proof}[Proof of Theorem~\ref{theorem: lim sol I}]
Assume that \( (q_\sigma^{(p)})_{\sigma \in \mathcal{B}_n} \) is such that
\begin{equation}\label{eq: our system and q}
\sum_{\sigma  \in \mathcal{B}_n} p^{\| \sigma_S\| } q_\sigma^{(p)} = \nu_p(1^S), \qquad S \subseteq [n]
\end{equation}
holds. Note that for \( p \) close to \( 1/2 \), we have that
\[
p^{\| \sigma_S\| } = 2^{-\| \sigma_S \|} + \| \sigma_S \| 2^{-\| \sigma_S \|+1} (p-1/2) + o(p-1/2).
\]
Further, as \( \nu_p \) is differentiable in \( p \) at \( 1/2 \), we have that
\[
\nu_p(1^{ S }) = \nu_{1/2}(1^{S })  + \nu'_{1/2}(1^{S })  (p-1/2) + o(p-1/2).
\]
Using these expansions, we will now apply \( \varphi \), as defined in Lemma~\ref{lemma: invertible operation}, to both sides of~\eqref{eq: our system and q}. To this end, we first introduce the following notation.
Given \( \sigma \in \mathcal{B}_n\) and \( S \subseteq [n]\),  write \( \sigma_S = \{ T_1, T_2, \ldots, T_m \} \), where \( m = \| \sigma \| \), to denote that the partition elements of \( \sigma \) when restricted to \( S \) are given by   \( T_1, T_2, \ldots, T_m \subseteq S \).  Using this notation, for any fixed  set \( S \subseteq [n] \) and \( \sigma \in \mathcal{B}_n \), we have that
\begin{align*}
&\sum_{S' \colon S' \subseteq S} (-2^{-1})^{|S| - |S'|} \cdot 2^{-\| \sigma_{S'}\| } 
 = \sum_{S_1, \ldots, S_m  \colon \atop  \forall i \in [m] \colon S_i \subseteq T_i } \prod_{i=1}^m (-2^{-1})^{|T_i| - |S_i|} 2^{-I(S_i \not = \emptyset)}
\\&\qquad  = \prod_{i=1}^m \sum_{S_i \colon S_i \subseteq T_i}  (-2^{-1})^{|T_i| - |S_i|} \cdot  2^{-I(S_i \not = \emptyset)}
 = \prod_{i=1}^m (1 + (-1)^{|T_i|} ) \cdot (2^{-1})^{|T_i|+1}
 \\&\qquad = 2^{-|S|} I(\sigma_S \text{ has only even sized partition elements}).
\end{align*}
Similarly, we have that
\begin{align*}
&\sum_{S' \colon S' \subseteq S} (-2^{-1})^{|S| - |S'|} \cdot   \| \sigma_{S'} \| \, 2^{-\| \sigma_{S'} \|+1 }  
\\&\qquad =
2  \cdot (-2)^{-|S|} \sum_{S' \colon S' \subseteq S} (-2)^{  |S'|} \cdot   \| \sigma_{S'} \|  \, 2^{-\| \sigma_{S'} \| }  
\\&\qquad =
2 \cdot (-2)^{-|S|} \sum_{I \colon I \subseteq [m]} |I|\,  2^{-|I|} \prod_{i \in I} \sum_{S_i \colon S_i \subseteq T_i,\atop S_i \not = \emptyset} (-2)^{|S_i|}
\\&\qquad =
2 \cdot (-2)^{-|S|} \sum_{I \colon I \subseteq [m]} |I| \, 2^{-|I|} \prod_{i \in I} ((1+(-2))^{|T_i|} - 1)
\\&\qquad =
2 \cdot (-2)^{-|S|} \sum_{I \colon I \subseteq [m]} |I| \, 2^{-|I|} \prod_{i \in I} \bigl( I(|T_i| \text{ is odd}) \cdot (-2) \bigr)
\\&\qquad =
2 \cdot (-2)^{-|S|} \sum_{I \colon I \subseteq [m]} |I| \, (-1)^{|I|}   I(|T_i| \text{ is odd} \text{ for all } i \in I)   
\\&\qquad =
2 \cdot (-2)^{-|S|} \cdot I(\sigma_S \text{ has exactly one odd sized partition element}) \cdot 1\cdot  (-1)^1
\\&\qquad =
2^{-|S|+1} \cdot I(\sigma_S \text{ has exactly one odd sized partition element}).
\end{align*}
Noting that  \( \varphi \), as defined in Lemma~\ref{lemma: invertible operation}, is linear, applying  it to~\eqref{eq: our system and q} and using the above derivations, we hence obtain
\begin{equation}
\begin{split}	
&\sum_{\sigma \in \mathcal{B}_n}
\Bigl(
2^{-|S|} I(\sigma_S \text{ has only even sized partition elements})
\\&\qquad\qquad + 2^{-|S|+1} \cdot I(\sigma_S \text{ has exactly one odd sized partition element}) (p-1/2)
\\&\qquad\qquad + o(p-1/2)
\Bigr) q_\sigma^{(p)}
\\&\qquad =
\sum_{S' \colon S' \subseteq S} (-2^{-1})^{|S| - |S'|} \Bigl(\nu_{1/2}(1^{S' })  + \nu'_{1/2}(1^{S' })  (p-1/2) + o(p-1/2)\Bigr).
\end{split}\label{eq: before limit}
\end{equation}
Using Lemmas~\ref{eq: odd rows are zero}~and~\ref{eq: even rows are zero}, it follows that this is equivalent to that
\begin{align*}
&\sum_{\sigma  \in \mathcal{B}_n}
\Bigl(
  I(\sigma_S \text{ has only even sized partition elements}) 
 + o(p-1/2)
\Bigr) q_\sigma^{(p)}
\\&\qquad =
  \sum_{S' \colon S' \subseteq S} (-2)^{|S'|} \nu_{1/2}(1^{S' }) + o(p-1/2)), \qquad  \text{if $|S|$ is even}
\end{align*}
and
\begin{align*}
&\sum_{\sigma  \in \mathcal{B}_n}
\Bigl(
 I(\sigma_S \text{ has exactly one odd sized partition element})  
 + o(1) 
\Bigr) q_\sigma^{(p)}
\\&\qquad =
 \sum_{S' \colon S' \subseteq S} (-2)^{ |S'|-1} \nu'_{1/2}(1^{S'})+  o(1), \qquad   \text{if $|S|$ is odd.}
\end{align*}
Now let \( (q_\sigma)_{\sigma \in \mathcal{B}_n} \) be any subsequential limit, as \( p \to 1/2 \), of formal solutions \( (q_\sigma^{(p)})_{\sigma \in \mathcal{B}_n}\) to~\eqref{eq: our system and q}. Then, combining the previous two equations and letting \( p \to 1/2 \), we obtain
\begin{align*}
&\sum_{ \sigma \in \mathcal{B}_n}
  I(\sigma_S \text{ has at most one odd sized partition element})  
\, q_\sigma
\\&\qquad =
\begin{cases}
  \sum_{S' \colon S' \subseteq S} (-2)^{|S'|} \nu_{1/2}(1^{S'})  &\text{ if $|S|$ is even} \cr
 \sum_{S' \colon S' \subseteq S} (-2)^{ |S'|-1} \nu'_{1/2}(1^{S'})    &\text{ if $|S|$ is odd.}
  \end{cases}
\end{align*}
By applying  \( \varphi^{-1} \) as defined in Lemma~\ref{lemma: invertible operation}, we obtain~\eqref{eq: limiting system}.
For the other direction, note that by Lemma~\ref{lemma: full rank matrix}, the   matrix  corresponding to the left hand side in the previous equation  has rank \( 2^n-n\). 
By Theorem~\ref{eq: marg.equal}, this is also the rank of \( A_{n,p} \) when \( p \not \in \{ 0,1/2,1 \} \), and hence of the equivalent matrix given by the left hand side of~\eqref{eq: before limit}. By a standard argument,  it follows that~\eqref{eq: limiting system} exactly describes the limiting solutions. 
This concludes the proof.
\end{proof}
 
We now provide the proof of Corollary~\ref{corollary: Ising}.

\begin{proof}[Proof of Corollary~\ref{corollary: Ising}]
We first need to place ourselves into the context of Theorem~\ref{theorem: lim sol I} which
we do as follows. With $J$ fixed, define a function $h$ from $(0,1)$ to $\mathbb{R}$
where $h(p)$ is such that the one-dimensional marginal of $\nu_{J,h(p)}$ is $p$.
It is easy to see that $h(1/2)=0$ and that $h$ is symmetric about $1/2$.
It also follows from  well known inequalities that $h$ is increasing, bijective and differentiable.
We now let $\nu_p \coloneqq \nu_{J,h(p)}$. Understanding what happens as $h\to 0$ is the same
as understanding what happens for $\nu_p$ as $p\to 1/2$. We need to look at the solutions
of~\eqref{eq: limiting system}. Only symmetric solutions can arise and we then, for a random
partition, let, for $i=1,2,3$, $q_i$ be the probability that there are $i$ partition elements.
$q_1$ and $q_3$ each correspond to one configuration while $q_2$ corresponds to three.
In~\eqref{eq: limiting system}, by symmetry, there are just four equations corresponding to $S$ having
sizes zero, one, two and three. $S$ having size zero and one both yield the equation
$$
q_1+q_2+q_3=1.
$$
The interesting equations are for $|S|$ being two and three. It is easy to check that the $|S|=2$ equation
yields
$$
\frac{q_1}{2}+\frac{q_2}{3}+\frac{q_3}{4}=\frac{e^{3J}+e^{-J}}{2e^{3J}+6e^{-J}}.
$$
For the $|S|=3$ equation, we first need the right hand side. By the chain rule,
this equals the derivative of the probability of having all 1's with respect to $h$ at $h=0$
times $h'(p)$ at $p=1/2$. For the latter, using the inverse instead, 
it is straightforward to compute $p'(h)$ at $h=0$ to be
$\frac{3e^{3J}+e^{-J}}{2e^{3J}+6e^{-J}}$, and hence 
$h'(1/2)=\frac{2e^{3J}+6e^{-J}}{3e^{3J}+e^{-J}}$. 
For the derivative of the probability of having all 1's with respect to $h$ for $h=0$,
a computation yields this to be $\frac{3e^{3J}}{2e^{3J}+6e^{-J}}$ and hence the right hand side
is $\frac{3e^{3J}}{3e^{3J}+e^{-J}}$. This easily yields the final equation to be
$$
q_1+q_2+\frac{3q_3}{4}=\frac{3e^{3J}}{3e^{3J}+e^{-J}}.
$$
One checks that the $3\times 3$ system has a unique solution and hence
Theorem~\ref{theorem: lim sol I} implies that 
$\lim_{h\to 0}q^{J,h}$ exists. One can also check that this unique solution is strictly positive
implying that for fixed $J$ and small $h$, $\nu_{J,h}$ is a color process.
One finds $q_2$ to be $\frac{12(e^{4J}-1)}{(3+e^{4J})(1+3e^{4J})}$ 
while one easily checks that $q_2^{\text{RCM}}=\frac{6e^{-2J}(e^{2J}-1)}{3+e^{4J}}$.
Since one can check that for all $J>0$, 
$\frac{12(e^{4J}-1)}{(3+e^{4J})(1+3e^{4J})} < \frac{6e^{-2J}(e^{2J}-1)}{3+e^{4J}}$,
we obtain the claim.
\end{proof}

\section*{Conflict of interest}

The authors declare that they have no conflict of interest.



\begin{thebibliography}{}

\bibitem{BMMU}
Bj\"ornberg, J.E., Mailler, C., M\"orters, P.,  Ueltschi, D.:
Characterising random partitions by random colouring. Electronic communications in probability, Volume 25, paper no. 4 (2020).

\bibitem{palo}
Forsstr\"om, M. P.:
Divide and color representations for Ising models with an external field. Preprint. 

\bibitem{PS.threshold}
Forsstr\"om, M. P., Steif, J. E.:
Divide and color representations for threshold Gaussian and stable vectors. Preprint. 
  
  \bibitem{st2017}
  Steif, J. E., Tykesson, J.:
  Generalized divide and color models.
ALEA. Latin American Journal of Probability and Mathematical Statistics,
{16}, pp. 899-955 (2019).
\end{thebibliography}
\end{document}